\documentclass[10pt]{amsart}
\usepackage{enumerate,amsmath,amssymb,latexsym,
amsfonts, amsthm, amscd}

\theoremstyle{plain}

\newtheorem{theorem}{Theorem}

\numberwithin{equation}{section}

\newcommand{\ra}{\rightarrow}

\newcommand{\mdot}{\,\begin{picture}(-1,-1)(-1,-1)\circle*{2}\end{picture}\ }

\addtocounter{section}{-1}

\begin{document}

\title {Painlev\'e IV: roots and zeros}

\date{}

\author[P.L. Robinson]{P.L. Robinson}

\address{Department of Mathematics \\ University of Florida \\ Gainesville FL 32611  USA }

\email[]{paulr@ufl.edu}

 \keywords{}

\begin{abstract}

We consider the (real) fourth Painlev\'e equation in which both parameters vanish, analyzing the square-roots of its solutions and paying special attention to their zeros. 

\end{abstract}

\maketitle

\medbreak

\section*{Introduction} 

In [2] we offered elementary proofs for fundamental properties of the unique triple-zero solution to the first Painlev\'e equation. In [3] we treated in a similar fashion all solutions to the second Painlev\'e equation whose graphs pass through the origin. Here we consider aspects of what is arguably the next case: the fourth Painlev\'e equation, which was discovered by Gambier. The general form of this equation is 
$$\frac{{\rm d}^2 w}{{\rm d} z^2} = \frac{1}{2 w} \Big(\frac{{\rm d} w}{{\rm d} z}\Big)^2 + \frac{3}{2} w^3 + 4 z w^2 + 2 (z^2 - \alpha) w + \frac{\beta}{w}$$
where $\alpha$ and $\beta$ are parameters. As is suggested by its form, this equation is properly {\it complex}: in fact, each of its solutions is meromorphic in the plane, with simple poles of residue $\pm 1$; see [1] and references therein. Note the presence of the dependent variable in denominators: this separates the fourth Painlev\'e equation from the first and second; of course, it engenders some complications. 

\medbreak 

In line with the setting of our previous papers, we shall consider the fourth Painlev\'e equation in purely real terms; moreover, we shall only consider the case in which $\alpha = 0$ and $\beta = 0$. Accordingly, our version of the fourth Painlev\'e equation (P IV) is 
$$\overset{\mdot \mdot}{s} = \frac{1}{2 s}\overset{\mdot}{s}^2 + \frac{3}{2} s^3 + 4 t s^2 + 2 t^2 s$$
to be solved for real $s$ as a function of real $t$; our preference here for $\overset{\mdot}{s}$ over $s'$ as notation for the derivative is largely on account of the otherwise awkward $s'^2$ or $(s')^2$ for its square. If the solution $s$ is strictly positive then its (positive) square-root $\sqrt s$ satisfies a second-order equation that is simpler than P IV in having no first-derivative term and no awkward denominators; in the opposite direction, the squares of nowhere-zero solutions to this simpler equation satisfy P IV. Those circumstances in which solutions to P IV or the simpler equation have (isolated) zeros call for separate handling. All of these matters are discussed under the sections {\it Square-roots} (on the differential equations themselves) and {\it Isolated zeros} (on the special handling of zeros); in our final section on {\it Remarks} we address some related issues without proof. 

\medbreak 

\section*{Square-roots} 

\medbreak 

We begin with some elementary observations regarding our version of the fourth Painlev\'e equation, which we restate for the record as 
\begin{equation} \label{P} 
\overset{\mdot \mdot}{s} = \frac{1}{2 s}\overset{\mdot}{s}^2 + \frac{3}{2} s^3 + 4 t s^2 + 2 t^2 s, \tag{{$\bf P$}}
\end{equation}
where the ratio on the right side is to be understood as a limit when necessary. 

\medbreak 

Observe that \ref{P} may be reformulated in a number of ways. First we may clear the awkward denominator, thus: 
\begin{equation*}  
2 s \overset{\mdot \mdot}{s} - \overset{\mdot}{s}^2  = 3 s^4 + 8 t s^3 + 4 t^2 s^2.  
\end{equation*} 
Further,  we may factor the right side, thus: 
\begin{equation*}  
2 s \overset{\mdot \mdot}{s} - \overset{\mdot}{s}^2  = s^2 (3s + 2 t)(s + 2 t).  
\end{equation*} 

\medbreak 

Observe also that reversing the sign of the dependent variable leads to the equation 
\begin{equation} \label{Pbar} 
\overset{\mdot \mdot}{s} = \frac{1}{2 s}\overset{\mdot}{s}^2 + \frac{3}{2} s^3 - 4 t s^2 + 2 t^2 s; \tag{{$\bf \overline{P}$}}
\end{equation} 
of course, sign-reversal in \ref{Pbar} leads to \ref{P} likewise.  Incidentally, passage between \ref{P} and \ref{Pbar} may also be effected by reversal of the independent variable. Of course, \ref{Pbar} admits reformulations akin to those for \ref{P} itself. 

\medbreak 

Let us agree to write $\mathbb {P}$ for the set comprising all solutions to the Painlev\'e equation \ref{P}; when extra clarity is called for, we may write $\mathbb{P} (I)$ for the set comprising all solutions to \ref{P} on the open interval (more generally, open set) $I \subseteq \mathbb{R}$. Similarly, we write $\overline{\mathbb{P}}$ for the set of all solutions to \ref{Pbar} (on some interval, which we may indicate for clarity). We observed above that multiplication by $-1$ yields a bijection 
$$\mathbb{P} \rightarrow \overline{\mathbb{P}} : s \mapsto - s;$$
also, that reversal of the independent variable yields a bijection from $\mathbb{P} (I)$ to $\overline{\mathbb{P}} (- I)$. 

\medbreak 

Now, let $s \in \mathbb{P}$ be a {\it strictly positive} solution to \ref{P} and write $\sigma : = \sqrt{s} = s^{1/2}$ for its positive square-root. Certainly, $\sigma$ is twice-differentiable. Further, from $s = \sigma^2$ it follows that 
$$\overset{\mdot}{s} = 2 \sigma \overset{\mdot}{\sigma}$$
so that 
$$\overset{\mdot}{\sigma}^2 = \frac{\overset{\mdot}{s}^2}{4 \sigma^2} =  \frac{\overset{\mdot}{s}^2}{4 s}$$
and 
$$\overset{\mdot \mdot}{s} = 2 \overset{\mdot}{\sigma}^2 + 2 \sigma \overset{\mdot \mdot}{\sigma} = \frac{\overset{\mdot}{s}^2}{2 s} + 2 \sigma \overset{\mdot \mdot}{\sigma}$$
so that 
$$\overset{\mdot \mdot}{s} - \frac{\overset{\mdot}{s}^2}{2 s} = 2 \sigma \overset{\mdot \mdot}{\sigma}.$$
All of this requires only that the twice-differentiable function $s$ be strictly positive. Recalling that $s$ is a solution to \ref{P} we deduce that 
$$2 \sigma \overset{\mdot \mdot}{\sigma} = \overset{\mdot \mdot}{s} - \frac{\overset{\mdot}{s}^2}{2 s} =  \frac{3}{2} s^3 + 4 t s^2 + 2 t^2 s = \frac{1}{2} s (3 s + 2 t)(s + 2 t)$$
or 
$$\overset{\mdot \mdot}{\sigma} = \frac{1}{4} \sigma (3 \sigma^2 + 2 t)(\sigma^2 + 2 t).$$

\medbreak 

This finding prompts us to formalize the auxiliary differential equation 
\begin{equation} \label{Phalf} 
4 \overset{\mdot \mdot}{\sigma} = \sigma (3 \sigma^2 + 2 t)(\sigma^2 + 2 t) \tag{{$\bf P^{1/2}$}}
\end{equation}
alongside its companion 
\begin{equation} \label{Pbarhalf} 
4 \overset{\mdot \mdot}{\sigma} = \sigma (3 \sigma^2 - 2 t)(\sigma^2 - 2 t). \tag{{$\bf \overline{P}^{1/2}$}}
\end{equation}
It also prompts us to introduce $\mathbb{P}^{1/2}$ and $\overline{\mathbb{P}}^{1/2}$ for the corresponding spaces of solutions. Notice that the map $\sigma \mapsto - \sigma$ preserves the spaces $\mathbb{P}^{1/2}$ and $\overline{\mathbb{P}}^{1/2}$ while the map $t \mapsto - t$ interchanges them. 

\medbreak 

The following result was established in the motivating lead-up to equation \ref{Phalf}.

\medbreak 

\begin{theorem} \label{sqrt}
If $s \in \mathbb{P}$ is strictly positive then $\sqrt{s} \in \mathbb{P}^{1/2}$. 
\end{theorem} 

\qed

\medbreak

As a companion result, if  $s \in \overline{\mathbb{P}}$ is strictly positive then a parallel argument places $\sqrt{s}$ in $\overline{\mathbb{P}}^{1/2}$; consequently, if $s \in \mathbb{P}$ is strictly {\it negative} then $\sqrt{-s} \in \overline{\mathbb{P}}^{1/2}$. 

\medbreak 

In the opposite direction, let $\sigma \in \mathbb{P}^{1/2}$ and write $s = \sigma^2$. Direct calculation as for the lead-up to \ref{Phalf} yields 
$$2 s \overset{\mdot \mdot}{s} - \overset{\mdot}{s}^2  = 4 \sigma^3 \overset{\mdot \mdot}{\sigma} = \sigma^4  (3 \sigma^2 + 2 t)(\sigma^2 + 2 t) = s^2 (3 s + 2 t)(s + 2 t).$$
If $\sigma$ is never zero then we may divide by $2 s$ throughout to see that $s \in \mathbb{P}$. We have established the following result. 

\medbreak 

\begin{theorem} \label{sq}
If $\sigma \in \mathbb{P}^{1/2}$ is nowhere zero then $\sigma^2 \in \mathbb{P}.$
\end{theorem}

\qed

\medbreak   

Similarly, if $\sigma \in \overline{\mathbb{P}}^{1/2}$ is nowhere zero then $\sigma^2 \in \overline{\mathbb{P}}$ and $- \sigma^2 \in \mathbb{P}$ is strictly negative. 

\medbreak 

Throughout the present section, we have deliberately avoided situations in which $\sigma \in \overline{\mathbb{P}}^{1/2}$ or $s \in \mathbb{P}$ has a zero. We shall address such situations carefully in the next section; naturally, we may ignore the identically zero function. 

\medbreak 

\section*{Isolated zeros} 

\medbreak 

As announced, we here consider situations in which $s \in \mathbb{P}$ or $\sigma \in \overline{\mathbb{P}}^{1/2}$ has a zero. Specifically, we shall assume that such a function has an isolated zero at the point $a$ in the open interval $I$: more specifically, we shall assume that the function vanishes at $a$ but at no other point of $I$. We wish to explore the extendibility of Theorem \ref{sqrt} and Theorem \ref{sq} to this context. 

\medbreak 

Observe at once from \ref{P} (say in a reformulation) that if $s \in \mathbb{P}$ satisfies $s(a) = 0$ then automatically $\overset{\mdot}{s}(a) = 0$. In like but more straightforward manner, \ref{Phalf} tells us that if $\sigma \in \mathbb{P}^{1/2}$ satisfies $\sigma(a) = 0$ then automatically $\overset{\mdot \mdot}{\sigma}(a) = 0$. We shall use these observations throughout the subsequent discussion, perhaps without comment. 

\medbreak 

Before proceeding further, it is convenient to draw attention to an important difference between \ref{P} and \ref{Phalf}. On the one hand, \ref{Phalf} has the form $\overset{\mdot \mdot}{\sigma} = \Phi(t, \sigma)$ in which $\Phi(t, \sigma)$ is a polynomial; consequently, the initial value problem for \ref{Phalf} has a standard local existence-uniqueness theorem. On the other hand, \ref{P} has the form $\overset{\mdot \mdot}{s} = F(t, s, \overset{\mdot}{s})$ in which $F(t, s, \overset{\mdot}{s})$ is rational but has $s$ in the denominator; the standard local existence-uniqueness theorem breaks down for initial data involving a zero of $s$. In fact, we have seen that if $s$ satisfies \ref{P} then the vanishing of $s(a)$ forces that of $\overset{\mdot}{s}(a)$; were standard local uniqueness to apply, a solution to \ref{P} with a zero would vanish throughout its interval of definition. 

\medbreak 

\begin{theorem} \label{NO} 
Let $\sigma \in \mathbb{P}^{1/2}$ and let $\sigma \geqslant 0$ on $I \ni a$. If $\sigma(a) = 0$ then $\sigma = 0$. 
\end{theorem} 

\begin{proof} 
The hypotheses ensure that not only $\sigma(a) = 0$ but also $\overset{\mdot}{\sigma}(a) = 0$. The identically zero function satisfies \ref{Phalf} on $I$ with the same initial data. The local uniqueness theorem for \ref{Phalf} now ensures that $\sigma = 0$. 
\end{proof}

\medbreak 

It follows at once that Theorem \ref{sqrt} has no direct extension allowing an isolated zero.

\medbreak 

\begin{theorem} \label{no} 
If $s \in \mathbb{P}$ is strictly positive except for an isolated zero at $a \in I$ then $\sqrt s \notin \mathbb{P}^{1/2}.$
\end{theorem} 

\begin{proof} 
The (positive) square-root $\sqrt{s}$ is zero at $a \in I$ but strictly positive on $I \setminus\{a\}$; Theorem \ref{NO} therefore excludes $\sqrt{s}$ from $\mathbb{P}^{1/2}$.
\end{proof} 

\medbreak 

Notwithstanding this negative result, we have the following. 

\medbreak 

\begin{theorem} \label{yes}
If $s \in \mathbb{P}$ is strictly positive except for an isolated zero at $a \in I$ then there exists $\sigma \in \mathbb{P}^{1/2}$ such that $s = \sigma^2$. 
\end{theorem} 

\begin{proof} 
To the left of $a$ there are only two continuous square-roots of $s$, namely $\pm \sqrt s$; likewise to the right of $a$. Since the taking of like signs on each side of $a$ leads to failure, we mix signs: thus, well-define $\sigma$ on $I$ by 
 \begin{equation*}
    \sigma(t)=
    \begin{cases}
      - \sqrt{s(t)} & \text{if}\ I \ni t \leqslant a, \\
      + \sqrt{s(t)} & \text{if} \ I \ni t \geqslant a.
    \end{cases}
  \end{equation*}
Theorem \ref{sqrt} easily places $\pm \sigma$ in $\mathbb{P}^{1/2}$ on $I\setminus\{a\}$; we must show that $\sigma$ is twice-differentiable at its zero $a$ with $\overset{\mdot \mdot}{\sigma}(a) = 0$. Let $I \ni t \neq a$: as $s = \sigma^2$, 
$$\overset{\mdot}{\sigma}(t) = \frac{\overset{\mdot}{s}(t)}{2 \sigma(t)}$$
while as $s \in \mathbb{P}$ and $s(a) = 0$, 
$$\frac{\overset{\mdot}{s}(t)^2}{2 s(t)} = \overset{\mdot \mdot}{s}(t) - \frac{1}{2} s(t)\Big(3 s(t) + 2 t \Big) \Big(s(t) + 2 t \Big)$$
and 
$$\lim_{t \ra a} \frac{\overset{\mdot}{s}(t)^2}{4 \sigma(t)^2} = \lim_{t \ra a} \frac{\overset{\mdot}{s}(t)^2}{4 s(t)} = \frac{1}{2}\overset{\mdot \mdot}{s}(a).$$
Checking signs, the taking of square-roots yields 
$$\lim_{t \ra a} \overset{\mdot}{\sigma}(t) = \sqrt{\frac{1}{2}\overset{\mdot \mdot}{s}(a)}\: ;$$
as $\sigma$ is continuous on $I$, it follows that $\sigma$ is (continuously) differentiable throughout $I$ by an application of the mean value theorem. As $\sigma$ satisfies \ref{Phalf} on $I\setminus\{a\}$, it follows that 
$$\lim_{t \ra a} \overset{\mdot \mdot}{\sigma}(t) = \lim_{t \ra a} \frac{1}{4} \sigma(t)\Big(3 \sigma(t)^2 + 2 t \Big) \Big( \sigma(t)^2 + 2 t \Big) = 0$$
whence a further application of the mean value theorem to the continuous function $\overset{\mdot}{\sigma}$ shows that $\sigma$ is twice-differentiable at $a$ with $\overset{\mdot \mdot}{\sigma}(a) = 0$ as required.
\end{proof} 

\medbreak 

Of course, a similar argument shows that if $s \in \mathbb{P}$ is strictly negative except for an isolated zero at $a \in I$ then there exists $\sigma \in \overline{\mathbb{P}}^{1/2}$ such that $s = - \sigma^2$; again, $\sigma$ takes opposite signs on opposite sides of $a.$ 

\medbreak 

Theorem \ref{no} and Theorem \ref{yes} are complements to Theorem \ref{sqrt} for cases in which $s \in \mathbb{P}$ has an isolated zero. There are analogous complements to Theorem \ref{sq} for cases in which $\sigma \in \mathbb{P}^{1/2}$ has an isolated zero. 

\medbreak 

The appropriate counterpart of Theorem \ref{no} is immediate. 

\medbreak 

\begin{theorem} \label{nono}
If $\sigma \in \mathbb{P}^{1/2}$ is strictly positive except for an isolated zero at $a \in I$ then $\sigma^2 \notin \mathbb{P}.$
\end{theorem} 

\begin{proof} 
If $\sigma^2$ were to lie in $\mathbb{P}$ then its positive square-root would lie outside $\mathbb{P}^{1/2}$ according to Theorem \ref{no}; but this positive square-root is $\sigma$ itself. 
\end{proof} 

\medbreak 

The appropriate counterpart of Theorem \ref{yes} requires just a little more work. 

\medbreak 

\begin{theorem} \label{yesyes}
If $\sigma \in \mathbb{P}^{1/2}$ takes opposite signs on opposite sides of $a \in I$ then $\sigma^2 \in \mathbb{P}.$
\end{theorem} 

\begin{proof} 
The twice-differentiable square $s := \sigma^2$ satisfies \ref{P} on $I\setminus\{a\}$ by Theorem \ref{sq}. 
Notice that if $t \in I$ then $\overset{\mdot}{s}(t) = 2 \sigma(t) \overset{\mdot}{\sigma}(t)$ and 
$$\overset{\mdot \mdot}{s}(t) = 2 \sigma(t) \overset{\mdot \mdot}{\sigma}(t) + 2 \overset{\mdot}{\sigma}(t)^2.$$
Consequently, as $\sigma(a) = 0$ it follows that 
$$\overset{\mdot \mdot}{s}(a) -  \frac{1}{2} s(a) (3s(a) + 2 a)(s(a) + 2 a) = 2 \overset{\mdot}{\sigma}(a)^2$$
and 
$$\lim_{t \ra a}\frac{\overset{\mdot}{s}(t)^2}{2 s(t)} = \lim_{t \ra a} \frac{(2 \sigma(t) \overset{\mdot}{\sigma}(t))^2}{2 \sigma(t)^2} = \lim_{t \ra a} 2 \overset{\mdot}{\sigma}(t)^2 = 2 \overset{\mdot}{\sigma}(a)^2.$$
This shows that $s$ satisfies \ref{P} at $a$ also and concludes the demonstration. 
\end{proof} 

\medbreak 

We close by remarking that the case of $s \in \mathbb{P}$ with an isolated zero at which the second derivative also vanishes is as tidy as can be: in fact, the case does not arise! Our first step towards this result is perhaps a little peculiar in hindsight. 

\medbreak 

\begin{theorem} \label{peculiar}
Let $s \in \mathbb{P}$ have an isolated zero at $a$. If $\overset{\mdot \mdot}{s}(a) = 0$ then each derivative of $s$ vanishes at $a$. 
\end{theorem} 

\begin{proof} 
The second reformulation of \ref{P} informs us that 
$$2 s \overset{\mdot \mdot}{s} - \overset{\mdot}{s}^2  = s^2 (3s + 2 t)(s + 2 t) = s^2 Q$$
say, where $Q$ is quadratic in $s$ and $t$. Away from the isolated zero, we may differentiate: the resulting terms $\pm 2 \overset{\mdot}{s} \overset{\mdot \mdot}{s}$ on the left cancel, to yield 
$$2 s \overset{\mdot \mdot \mdot}{s} = 2 s \overset{\mdot}{s} Q + s^2 \overset{\mdot}{Q}$$
whence 
$$2 \overset{\mdot \mdot \mdot}{s} = 2 \overset{\mdot}{s} Q + s \overset{\mdot}{Q}$$
away from $a$ and hence at $a$ also. All that remains is to differentiate inductively. 
\end{proof} 

\medbreak 

We can now see that this case is indeed vacuous: taking the ({\it difficult} !) meromorphicity of $s$ for granted, the identity theorem implies that $s$ is zero throughout the open interval in which $a$ is an isolated zero; this is absurd!  

\medbreak 

In particular, it follows that $s \in \mathbb{P}$ cannot change sign at an isolated zero.

\medbreak 

\section*{Remarks} 

\medbreak 

We round out our account with some miscellaneous comments on related topics of interest. 

\medbreak 

Recall that the fourth Painlev\'e equation is properly a complex differential equation. The process of passing to a square-root is naturally more elaborate in the complex setting: as we mentioned, solutions to the fourth Painlev\'e equation are meromorphic, with {\it simple} poles; square-roots of such functions cannot be meromorphic! Nonetheless, there is sufficient reason for further study of the relevant auxiliary equation 
$$4 \frac{{\rm d}^2 \omega}{{\rm d} z^2} = \omega (3 \omega^2 + 2 z) (\omega^2 + 2 z).$$

\bigbreak 

We have seen in our study of the fourth Painlev\'e equation \ref{P} that the auxiliary differential equation \ref{Phalf} is of definite theoretical interest. In fact, this auxiliary equation is also of considerable practical help, aside from its ability to handle initial data involving a zero. We began exploring solutions of the fourth Painlev\'e equation with the aid of WZGrapher, a valuable freeware program developed by Walter Zorn. Quite early in our explorations, we noticed apparent graphical instabilities: for example, solutions of \ref{P} with certain initial data would at first appear to be oscillatory; upon zooming out, such a solution might seem to suffer a catastrophe, oscillations disappearing and being replaced by a blow-up or spike; upon zooming out further, oscillations might reappear; and so on. Not surprisingly, such catastrophic behaviour manifests itself at a zero of the solution and so involves the awkward denominator in \ref{P}. These apparent graphical instabilities seem to be removed by passage to the corresponding solutions of \ref{Phalf}, as the reader may care to see using WZGrapher. 

\medbreak 

The factorized form of the fourth Painlev\'e equation \ref{P}:  
$$\overset{\mdot \mdot}{s} - \frac{\overset{\mdot}{s}^2}{2 s}  = \frac{1}{2} s (3s + 2 t)(s + 2 t)$$
indicates that the lines `$s = 0$', `$s = -2 t/3$' and `$s = - 2 t$' have geometric significance for its solutions. Similarly, `$\sigma = 0$' and the parabolas `$\sigma^2 = - 2 t/3$' and `$\sigma^2 = - 2 t$' have geometric significance for solutions to the auxiliary equation \ref{Phalf}: 
$$4 \overset{\mdot \mdot}{\sigma} = \sigma (3 \sigma^2 + 2 t)(\sigma^2 + 2 t). $$
This geometric significance can be seen in a concavity diagram. The curves `$\sigma = 0$', `$\sigma^2 = - 2 t/3$' and `$\sigma^2 = - 2 t$' divide the $(t, \sigma)$-plane into regions. The sign of $\overset{\mdot \mdot}{\sigma}$ is negative/positive in the regions directly above/below the half-parabola `$\sigma = + \sqrt{- 2 t/3}$' so that solutions to \ref{Phalf} have the opportunity to oscillate about this half-parabola; similarly, solutions to \ref{Phalf} may oscillate about `$\sigma = - \sqrt{- 2 t/3}$'. 

\medbreak 

In fact, experimentation with WZGrapher reveals that solutions to \ref{Phalf} that do not suffer blow-up in both time directions tend to display steadily decaying oscillations about the upper or lower half of the parabola `$\sigma^2 = - 2 t/3$' as $t \ra - \infty$; and that solutions often tend to linger alongside `$\sigma = 0$' or `$\sigma^2 = - 2 t$' as they make more or less extended approaches to tangency. Also, it not infrequently happens that a minuscule change in initial data causes a solution $\sigma$ to flip its oscillations from one half-parabola to the other, or to flip the direction of its finite-time blow-up, in such a way that the sudden transition is not detectable in $\sigma^2$. On a more aesthetic note, when oscillations of $\sigma \in \mathbb{P}^{1/2}$ occur about a half-parabola `$\sigma = \pm \sqrt{- 2 t/3}$' they are quite evenly balanced. By contrast, when oscillations of $s \in \mathbb{P}$ occur about the line `$s = -2 t/3$' they are uneven, displaying larger arches on the side of the line away from `$s = 0$'. Of course, squaring accounts for the difference.

\medbreak 

One relatively simple family of illustrative examples takes $\sigma \in \mathbb{P}^{1/2}$ with $\sigma(0) = 0$ and $\overset{\mdot}{\sigma}(0)$ strictly positive. As $\overset{\mdot}{\sigma}(0)$ increases from $0$ to a little beyond $1.169868591$, two gradual changes to the solution $\sigma$ take place simultaneously: on the one hand, $\sigma$ oscillates about `$\sigma = - \sqrt{- 2 t/3}$', the amplitude of the oscillations initially decreasing and finally increasing; on the other hand, $\sigma$  lingers initially along `$\sigma = 0$' and finally along `$\sigma = - \sqrt{- 2 t}$'; when $\overset{\mdot}{\sigma}(0)$ is around $0.65$ the oscillations have their least amplitude and there is no lingering along either curve. As $\overset{\mdot}{\sigma}(0)$ increases from $1.169868591$ to $1.169868592$ the oscillations disappear, to be replaced by a negative blow-up in finite negative time; thereafter, the lingering along `$\sigma = - \sqrt{- 2 t}$' gradually disappears and the finite-time blow-up accelerates. Throughout, $\sigma \in \mathbb{P}^{1/2}$ has a unique zero, at which it changes sign; accordingly, its square lies in $\mathbb{P}$. 

\medbreak 

We leave to the reader the pleasure of exploring this family of examples in WZGrapher (or some similar program). Among many other families to explore, we recommend the following: take $\sigma(0) = 1$ and let $\overset{\mdot}{\sigma}(0)$ run from $-0.933899363$ to $1.579186627$, noting the several transitions with reference to `$\sigma = 0$', `$\sigma^2 = - 2 t/3$' and `$\sigma^2 = - 2 t$'; take $\sigma(-6) = 2$ and let $\overset{\mdot}{\sigma}(-6)$ run from $-0.170889967$ to $-0.170889968$ (!).  

\bigbreak

\begin{center} 
{\small R}{\footnotesize EFERENCES}
\end{center} 
\medbreak 

[1] V.I. Gromak, I. Laine and S. Shimomura, {\it Painlev\'e Differential Equations in the Complex Plane}, de Gruyter (2002). 

\medbreak 

[2] P.L. Robinson, {\it The Triple-Zero Painlev\'e I Transcendent}, arXiv 1607.07088 (2016). 

\medbreak 

[3] P.L. Robinson, {\it Homogeneous Painlev\'e II Transcendents}, arXiv 1608.02139 (2016). 
\medbreak

\end{document}